\documentclass[letterpaper, 10 pt, conference]{ieeeconf}  
\usepackage{useful_math, comment, cite}

\IEEEoverridecommandlockouts                              

\title{\LARGE \bf
Finding Super-spreaders in SIS Epidemics
}

\author{Anirudh Sridhar$^{1}$ and Arnob Ghosh$^{2}$
\thanks{$^{1}$Anirudh Sridhar is with the Department of Electrical \& Computer Engineering, New Jersey Institute of Technology
        {\tt\small anirudh.sridhar@njit.edu}}%
\thanks{$^{2}$Arnob Ghosh is with the Department of Electrical \& Computer Engineering, New Jersey Institute of Technology
        {\tt\small arnob.ghosh@njit.edu}}%
}

\begin{document}

\maketitle
\thispagestyle{empty}
\pagestyle{empty}

\begin{abstract}
In network epidemic models, controlling the spread of a disease often requires targeted interventions such as vaccinating high-risk individuals based on network structure.
However, typical approaches assume complete knowledge of the underlying contact network, which is often unavailable. While network structure can be learned from observed epidemic dynamics, existing methods require long observation windows that may delay critical interventions.

In this work, we show that full network reconstruction may not be necessary: control-relevant features, such as high-degree vertices (super-spreaders), can be learned far more efficiently than the complete structure.
Specifically, we develop an algorithm to identify such vertices from the dynamics of a Susceptible-Infected-Susceptible (SIS) process. 
We prove that in an $n$-vertex graph, vertices of degree at least $n^\alpha$ can be identified over an observation window of size $\Omega (1/\alpha)$, for any $\alpha \in (0,1)$.
In contrast, existing methods for exact network reconstruction requires an observation window that grows linearly with $n$.
Simulations demonstrate that our approach accurately identifies super-spreaders and enables effective epidemic control.
\end{abstract}

\section{Introduction}

The dynamics of viral spread in large populations depend not only on biological factors but also on the underlying \emph{contact network} through which diseases propagate \cite{newman2002spread, salathe2010resolution}. The role of the network in particular makes the design of interventions such as vaccination or immunization challenging, as the benefit of protecting an individual depends critically on who they interact with in the network.
As such, existing approaches for epidemic control typically require full knowledge of the contact network; see, e.g., 
\cite{borgs2010antidote, preciado2014optimal, nowzari2016analysis, zhang2024estimation}.
Such methods are well-suited for meta-population models, where nodes represent large aggregated groups, but are less applicable to individual-level viral spread, where fine-grained contact patterns are rarely observed.

To get around the absence of the contact network, a natural solution is to first reconstruct the network from data, and then apply interventions to the learned network. The task of reconstructing the network from data, known as \emph{structure learning}, has received considerable attention from both theoretical and practical perspectives \cite{WT04_epidemic,AA05_blog, gomez2012inferring,ACFKP13_trace_complexity,shen2014reconstructing, prasse2019infeasible, elahi2025learn}. However, such approaches tend to be computationally demanding or data-intensive, which may delay the implementation of time-sensitive interventions.

In this paper, we propose an alternative path forward. Instead of reconstructing the entire network, we estimate high-degree vertices (or super-spreaders), whose removal from a network can dramatically change epidemic dynamics \cite{pastor2002immunization, dezso2002halting}.
In \cite{mossel2024finding, brandenberger2025detecting}, this idea was recently explored in the context of Susceptible-Infected (SI) epidemics; here we focus on the Susceptible-Infected-Susceptible (SIS) process, which is more appropriate for endemic diseases and recurring outbreaks since it models re-infections \cite{contact_process, brauer2012mathematical}. 
We prove that for any $\alpha \in (0,1)$, vertices of degree at least $n^\alpha$ in a $n$-vertex network can be identified in $T = \Omega (1 / \alpha)$ time (see Theorem \ref{thm:time_horizon} for a more precise statement). 
In contrast, the best current theoretical guarantees for reconstructing sparse networks from SIS process observations require $T = \Omega(n)$ \cite{elahi2025learn}. We then investigate the empirical accuracy of our estimation procedure and its implications for control through simulations. 

\emph{Related work.} 
The structure learning problem for spreading processes has largely been studied in the context of Susceptible-Infected (SI) or Susceptible-Infected-Recovered (SIR) models \cite{WT04_epidemic, AA05_blog, gomez2012inferring}.
In such models, since nodes are only infected once before recovering, data from multiple distinct epidemics are used. In models with re-infection such as the Susceptible-Infected-Susceptible (SIS) process, it is possible to learn the underlying network from a single epidemic trajectory, though additional temporal correlations arise as a result \cite{shen2014reconstructing, prasse2019infeasible, elahi2025learn}. 
The best known theoretical results on exact structure learning for SIS epidemics requires one to observe the epidemic for $T = \Omega(n)$ time if the underlying graph has a bounded maximum degree as $n$ grows large, and in general $T = \exp ( \Omega(n))$ is needed \cite{elahi2025learn}.

Since fully learning networks precisely requires a substantial amount of data, a recent line of work \cite{mossel2024finding, brandenberger2025detecting} has explored whether one can instead learn key network features -- such as high-degree vertices --  using much less information. The aforementioned work shows that in the SI process, vertices of degree larger than $\sqrt{n}$ can be estimated with minimal data (i.e., from a constant number of epidemic trajectories), while learning vertices of degree smaller than $\sqrt{n}$ requires nearly as much data as learning the entire network. 
Interestingly, we show that this phase transition is non-existent in the SIS model: re-infections allow for any vertex of degree $n^\alpha$ (for any $\alpha \in (0,1)$) to be quickly estimated. 
Our algorithm for estimating high-degree vertices is also a significant departure from the methods of \cite{mossel2024finding}.

\emph{Notation.} 
For a set $\cA$, we let $|\cA|$ denote its cardinality. We represent a graph by $G = (V,E)$, where $V$ is the vertex set and $E$ is the edge set. We denote $n : = |V|$ and for a vertex $v \in V$, $\cN(v)$ denotes the set its neighbors.
Throughout, we use standard asymptotic notation, i.e., $O(\cdot), o(\cdot), \Omega(\cdot), \Theta(\cdot)$, and all limits are with respect to $n \to \infty$ unless otherwise specified.

\emph{Organization.}
We discuss the mathematical setup, our main results, and our algorithm in Section \ref{sec:results}. Full proofs of our main results are contained in Section \ref{sec:proofs}.
Section \ref{sec:simulations} discusses our simulation results, and we conclude in Section \ref{sec:conclusion}.

\section{Definitions and main results}
\label{sec:results}

\subsection{The SIS model}

We assume the infection process can be modeled by a stochastic Susceptible-Infected-Susceptible (SIS) process \cite{contact_process, brauer2012mathematical}, which is perhaps the simplest model for viral spread on networks with re-infections. To define the process, let $G$ be a fixed graph upon which the infection spreads, let $V$ be the vertex set of the graph, let $\beta > 0$ be the (pairwise) infection rate, and let $\gamma > 0$ be the node-level recovery rate. At any given point in time, a node $v$ is either susceptible or infected. Formally, we denote $\cS(t)$ and $\cI(t)$ to be the set of susceptible and infected nodes in the graph at time $t$. 

If a node is infected at time $t$, it recovers at rate $\gamma$. In other words, the time it takes for an infected vertex to become susceptible again is a $\Exp(\gamma)$ random variable.
If a node is susceptible at time $t$, the rate at which the node is infected is equal to the number of infected neighbors. That is, for a vertex $v \in \cS(t)$, 
\begin{equation}
\label{eq:infection_rate}
\p ( v \in \cI(t + \epsilon) \vert \cF_t) = \epsilon \beta | \cN(v) \cap \cI(t) | + o(\epsilon),
\end{equation}
where $o(\epsilon) \to 0$ faster than $\epsilon \to 0$, and $\{ \cF_t \}_{t \ge 0}$ is the natural filtration corresponding to the stochastic evolution of the cascade. 

\begin{assumption}
We assume that the initial distribution of infected and susceptible nodes in $G$ are arbitrary.
Importantly, while the parameters $\beta, \gamma$ as well as the network $G$ are unknown, we assume access to infection and recovery times for each node in some observation interval $[0,T]$ (though we emphasize that the epidemic may evolve in a manner unknown to the observer before and after this interval).
\end{assumption}

Since our goal is to identify high-degree vertices, we must precisely define what we mean by ``high-degree'' and ``low-degree''. We do so in the following assumption. 

\begin{assumption}
\label{as:graph}
Let $m, d, D$ be positive integers.
Let $G$ be a graph on $n$ vertices such that there are at most $m$ vertices of degree at least $D$, and all other vertices have degree at most $d$. 
We assume that as $n \to \infty$, $m$ is constant while $d = n^{o(1)}$, and $D = n^\alpha$ for some $\alpha \in (0,1)$.
\end{assumption}

We remark that Assumption \ref{as:graph} was also used in prior work on detecting and estimating high-degree vertices in the SI process \cite{mossel2024finding, brandenberger2025detecting}. In essence, the assumption states that high-degree nodes are relatively rare but are quite influential compared to most nodes.

\subsection{Estimating high-degree vertices}

Our first main result shows that, once a vertex has been infected sufficiently many times, it can be accurately classified as low-degree or high-degree.

\begin{theorem}
\label{thm:reinfections}
Let $T > 0$ be a time horizon and let $K$ be a positive integer. Let us additionally define $V_K(T)$ to be the set of vertices that have been infected more than $K$ times in $[0,T]$.
Then there is an algorithm (depending on $K$) which outputs a set of estimated high-degree vertices $\widehat{\mathrm{HD}}$ such that, if $K > 2/\alpha$, then
\[
\widehat{\mathrm{HD}} = \{ v  \in V_K(T) : \deg(v) \ge n^\alpha \},
\]
with probability $1 - o(1)$.
\end{theorem}

In other words, it is possible to determine whether a vertex is high-degree or not after it has been infected a constant number of times, though smaller values of $\alpha$ necessitate a larger number of infections per node.

Our second main result provides a concrete estimate for how long the system should be observed to accurately estimate high-degree vertices.

\begin{theorem}
\label{thm:time_horizon}
Let $\delta > 0$, let $K = \lceil 3 / \alpha \rceil$, and let $\deg(v) \ge n^\alpha$.
Suppose that $T = \Omega (\alpha^{-1} \log (1 / \delta))$ and that $v$ is infected in the interval $[0,1]$. Then $\p ( v \in \widehat{\mathrm{HD}} ) \ge 1 -\delta$, where $\widehat{\mathrm{HD}}$ is the same estimator from Theorem \ref{thm:reinfections}.
\end{theorem}

An important takeaway from the theorem is that high-degree vertices can be accurately estimated after observing the system for a constant amount of time (with respect to $n$). In contrast, the best known results on full network reconstruction require $T = \Omega(n)$ if the underlying graph is sparse, or $T = \exp ( \Omega(n))$ in general \cite[Theorem 3.1]{elahi2025learn}. This drastic comparison highlights the advantage of learning key features rather than the entire network.
We also remark that the observation time $T$ grows as $\alpha$ and $\delta$ decrease, which reflects weaker signals and higher accuracy requirements. Moreover, if the epidemic never reaches a high-degree vertex $v$, no information about $v$ can be learned. The assumption that $v$ is infected in $[0,1]$ ensures that sufficient activity at node $v$ occurs in the observation window.
We finally remark that while we set a concrete value of $K$ here, any choice of $K = \Theta (1/\alpha)$ larger than $2/\alpha$ proves the claim.

\begin{remark}[Implications for epidemic control]
Epidemics can be mitigated through targeted node-level interventions, such as vaccinations, which prevent nodes from transmitting a disease to its neighbors.
Targeting high-degree nodes is particularly effective \cite{borgs2010antidote, pastor2002immunization, dezso2002halting}.
If the underlying graph is unknown, one can instead target the set of estimated high-degree nodes, given by $\widehat{\mathrm{HD}}$ (see Theorem \ref{thm:reinfections} and \ref{thm:time_horizon}).
If high-degree vertices are successfully targeted through $\widehat{\mathrm{HD}}$, the epidemic may even be driven to extinction in certain cases \cite{pastor2002immunization, dezso2002halting}.
\end{remark}

\subsection{The estimation algorithm}

The key insight behind our algorithm is that high-degree vertices get re-infected at a much faster rate than low-degree vertices. For a vertex of degree $d$, the re-infection rate is at most $\beta d$, and may be even lower if infected neighbors recover.
In contrast, when a high-degree vertex becomes infected, much of its neighborhood quickly follows, and many neighbors remain infected when $v$ recovers. Consequently, $v$ is re-infected extremely quickly, creating an ``echo chamber'' that drives rapid re-infection.

Our algorithm leverages this idea by measuring the re-infection time for each vertex; that is, the amount of time it takes for a vertex to become infected again after recovering. While the nodes with the smallest re-infection times are most likely to have high degrees, there may also be false positives due to the stochastic nature of the epidemic. 
To mitigate this issue, our algorithm outputs a set of vertices which have been infected at least $K$ times and consistently have a small re-infection time. For the right choice of $K$ (i.e., $K > 2/\alpha$), the probability of false positives becomes negligible.

We now present the full details of our algorithm. For a vertex $v$ and a positive integer $k$, let $I_k(v)$ denote the $k$th time node $v$ becomes infected. Let $S_k(v)$ denote the $k$th time node $v$ becomes susceptible after infection (so that $S_1(v) \ge I_1(v)$ even if $v$ is initially susceptible). The \emph{$k$th time until re-infection} is given by $I_{k + 1}(v) - S_k(v)$.
Given a threshold $h$ and a re-infection number $K$, our algorithm (Algorithm \ref{alg:reinfections}) outputs the set of vertices for which the first $K$ re-infection times are at most $h$.

\begin{breakablealgorithm}
\caption{Estimating high-degree vertices}
\label{alg:reinfections}
\begin{algorithmic}[1]
\Require{Sequence of infection and recovery events per node in the interval $[0,T]$, threshold $h >0$, re-infection number $K$.}
\Ensure{A set $\widehat{\mathrm{HD}} \subset V$.}
\State For each $v \in V$ that is infected more than $K$ times in $[0,T]$ (i.e., $I_{K+1}(v) \le T$), compute
\begin{align}
\label{eq:reinfection_time}
R_K(v) & : = \max_{1 \le k \le K} ( I_{k + 1}(v) - S_k(v)).
\end{align}
\State Return $\widehat{\mathrm{HD}} : = \{v : R_K(v) \le h \}$.
\end{algorithmic}
\end{breakablealgorithm}

In Section \ref{sec:proofs}, we analyze the performance of Algorithm \ref{alg:reinfections} and show that setting $K = \lceil 3 / \alpha \rceil$ and $h = n^{- \alpha/2}$ (assuming $\alpha$ is indeed known) yields an estimate with the properties discussed in Theorems \ref{thm:reinfections} and \ref{thm:time_horizon}.

\section{Analysis of re-infection times: Proofs of Theorems \ref{thm:reinfections} and \ref{thm:time_horizon}}
\label{sec:proofs}

\subsection{Re-infection time of a high-degree vertex}

The main technical result behind the proofs of our theorems bounds the re-infection time of a high-degree vertex.

\begin{lemma}
\label{lemma:reinfection_high_degree}
If $\deg(v) \ge n^\alpha$, then $I_2(v) - S_1(v) \le n^{- \alpha + o(1)}$ with probability $1 - o(1)$, where $o(1) \to 0$ as $n \to \infty$.
\end{lemma}

To prove the lemma, we start with the following supporting result, which lower bounds the number of infected neighbors in a high-degree vertex's neighborhood when it recovers.

\begin{lemma}
\label{lemma:sis_neighborhood_1}
If $\deg(v) \ge n^\alpha$, then $| \cN(v) \cap \cI( S_1(v)) | \ge n^{\alpha - o(1)}$ with probability $1 - o(1)$.
\end{lemma}

\begin{proof}
Denote $\cA : = \cN(v) \cap \cS ( I_1(v))$ and $\cB : = \cN(v) \cap \cI(I_1(v))$ to be the set of susceptible and infected nodes in the neighborhood of $v$, respectively, when $v$ is first infected. As a shorthand, we will also denote $\tau : = S_1(v) - I_1(v)$, which is a $\mathrm{Exp}(\gamma)$ random variable.

Note that a vertex $u \in \cA$ is infected at time $S_1(v)$ if (1) $u$ is infected before $v$ recovers and (2) the recovery time of $u$ exceeds $\tau$. Conditioned on $\tau$, the probability of the first scenario is at least $1 - e^{- \beta \tau}$, since $u$ is adjacent to $v$ but may have more infected neighbors. The conditional probability of the second scenario is equal to $e^{- \gamma \tau}$. Consequently, conditioned on $\tau$, the number of vertices in $\cA$ that are infected at time $S_1(v)$ is stochastically lower bounded by $X \sim \mathrm{Bin}( |\cA|, ( 1 - e^{- \beta \tau}) e^{- \gamma \tau})$.

On the other hand, a vertex $w \in \cB$ is infected at time $S_1(v)$ if its recovery time is at least $\tau$, which happens with probability $e^{- \gamma \tau}$ conditioned on $\tau$. Consequently, conditioned on $\tau$, the number of vertices in $\cB$ that are infected at time $S_1(v)$ is stochastically lower bounded by $Y \sim \mathrm{Bin} ( |\cB|, e^{- \gamma \tau})$.

Since $X$ and $Y$ are conditionally independent, $X + Y$ stochastically bounds a  $\mathrm{Bin}(\deg(v), (1 - e^{- \beta \tau}) e^{-\gamma \tau})$ random variable. 
If $( \log n)^{-1/2} \le \tau \le (\log n)^{1/2}$, then  $(1 - e^{-\beta \tau}) e^{-\gamma \tau} = n^{-o(1)}$. Consequently, Hoeffding's inequality implies that $X + Y \ge n^{\alpha - o(1)}$ with probability $1 - o(1)$. The claimed result follows since $\tau$ falls in the aforementioned range with probability $1 - o(1)$ and since $X + Y$ is stochastically bounded by $| \cN(v) \cap \cI( S_1(v)) |$. \qed
\end{proof}

We are now ready to prove Lemma \ref{lemma:reinfection_high_degree}.

\begin{proof}[Proof of Lemma \ref{lemma:reinfection_high_degree}]
Condition on $\cF_{S_1(v)}$ (which is valid since $S_1(v)$ is an almost surely finite stopping time) and assume that $| \cN(v) \cap \cI( S_1(v)) | \ge n^{\alpha - o(1)}$.
The time it takes for a given $u \in \cN(v) \cap \cI(S_1(v))$ to recover and to infect $v$ are independent $\Exp(\gamma)$ and $\Exp(\beta)$ random variables. The probability that $u$ is infected from time $S_1(v)$ to $I_2(v)$ (that is, until $v$ is re-infected) is at least the probability that $u$ infects $v$ before recovering, which is $\beta / (\beta + \gamma)$. 
Letting $Z$ be the number of vertices in $\cN(v)$ that stay infected from time $S_1(v)$ to $I_2(v)$, we can stochastically lower bound $Z$ by a $\Bin(n^{\alpha - o(1)}, \beta / (\beta + \gamma))$ random variable (this follows from the independence of recovery and pairwise infection times). By Hoeffding's inequality, it follows that for any $\epsilon > 0$, $Z \ge  n^{\alpha -\epsilon}$ with probability at least $1 - o(1)$.

Next, note that on the event where $Z \ge n^{\alpha - \epsilon}$, the re-infection time of $v$ is stochastically bounded by a $\Exp( n^{\alpha - \epsilon})$ random variable. 
By Markov's inequality, this random variable is at most $n^{-\alpha + 2 \epsilon}$ with probability $1 - o(1)$.

Together, our arguments show that, conditioned on $\cF_{S_1(v)}$ and assuming $| \cN(v) \cap \cI(S_1(v)) | \ge n^{\alpha - o(1)}$, we have $I_2(v) - S_1(v) \le n^{-\alpha + 2 \epsilon}$ with probability $1 - o(1)$.
Let us note that since $| \cN(v) \cap \cI(S_1(v)) | \ge n^{\alpha - o(1)}$ with probability $1 - o(1)$ by Lemma \ref{lemma:sis_neighborhood_1}, we may remove the conditioning on $\cF_{S_1(v)}$ to show \emph{unconditionally} that $I_2(v) - S_1(v) \ge n^{-\alpha + 2 \epsilon}$ with probability $o(1)$.
As this holds for any $\epsilon > 0$, we may take $\epsilon \to 0$ at an arbitrarily slow rate to prove the claimed result. \qed
\end{proof}

\subsection{Multiple re-infections: Proof of Theorem \ref{thm:reinfections}}

Recall that for a vertex $v$ and a positive integer $K$, $R_K(v) = \max_{1 \le k \le K} ( I_{k + 1}(v) - S_k(v))$ is the maximum of the first $K$ re-infection times of $v$. 
We also recall that $V_K(T)$ is the set of vertices that have experienced more than $K$ infections in $[0,T]$.
We prove Theorem \ref{thm:reinfections} by showing that $R_K(v)$ exhibits different behaviors for high-degree and low-degree $v$, provided $K$ is large enough.

\begin{lemma}
\label{lemma:RK_high_degree}
Let $K$ be a fixed positive integer. If $\deg(v) \ge n^\alpha$, then $R_K(v) \le n^{- \alpha + o(1)}$ with probability $1 - o(1)$, where $o(1) \to 0$ as $n \to \infty$.
\end{lemma}

\begin{proof}
For $\theta < \alpha$, we can bound
\begin{align*}
& \p \left ( R_K(v) > n^{-\theta} \right)  \stackrel{(a)}{\le} \sum_{k = 1}^K \p \left ( I_{k + 1}(v) - S_k(v) > n^{-\theta} \right) \\
& \hspace{0.5cm} \stackrel{(b)}{=} \p \left ( I_2(v) - S_1(v) > n^{-\theta} \right ) \\
& \hspace{1cm} + \sum_{k = 2}^K \E \left[ \p \left( I_{k + 1}(v) - S_k(v) > n^{-\theta} \vert \cF_{I_k(v)} \right) \right],
\end{align*}
where the $(a)$ is due to a union bound, and the $(b)$ uses the tower rule for conditional expectations. By Lemma \ref{lemma:reinfection_high_degree}, the probability corresponding to $I_2(v) - S_1(v)$ is $o(1)$. For each term in the sum, we note that the conditional probabilities are each $o(1)$ by Lemma \ref{lemma:reinfection_high_degree} as well, since the $k$th re-infection time $I_{k + 1}(v) - S_k(v)$ can be interpreted as the first re-infection time, if we start the SIS process at time $I_k(v)$. Consequently, since $K$ is assumed to be constant with respect to $n$, $R_K(v) > n^{- \theta}$ with probability $o(1)$. As this holds for all $\theta < \alpha$, the claimed result follows. \qed
\end{proof}

\begin{lemma}
\label{lemma:RK_low_degree}
Let $K > 2/\alpha$ be an integer. If $\deg(v) \le d$ (with $d$ defined in Assumption \ref{as:graph}), then $R_K(v) \le n^{- \alpha/2}$ with probability $o ( 1/n)$.
\end{lemma}

\begin{proof}
If $\deg(v) \le d$, then each re-infection time $I_{k + 1}(v) - S_k(v)$ stochastically upper bounds a $\Exp(  d \beta)$ random variable; moreover, these random variables are independent. Consequently, 
\begin{multline*}
\p \left( \max_{1 \le k \le K} ( I_{k + 1}(v) - S_k(v) ) \le n^{- \alpha / 2} \right) \\
\le \left ( 1 - \exp ( d \beta n^{- \alpha / 2} ) \right)^K \stackrel{(a)}{\le} ( d \beta n^{- \alpha / 2} )^K \stackrel{(b)}{=} o \left( \frac{1}{n} \right),
\end{multline*}
where $(a)$ uses the inequality $1 - e^{-x} \le x$, and $(b)$ holds since $d \beta = n^{o(1)}$ and $K > 2 / \alpha$. \qed
\end{proof}

Together, the two lemmas readily prove Theorem \ref{thm:reinfections}.

\begin{proof}[Proof of Theorem \ref{thm:reinfections}]
We first note that since the number of high-degree vertices is constant with respect to $n$ (see Assumption \ref{as:graph}), Lemma \ref{lemma:RK_high_degree} implies that with probability $1 - o(1)$, 
\begin{equation}
\label{eq:RK_high_degree}
\min_{v \in V_K(T) : \deg(v) \ge n^\alpha} R_K(v) \le n^{- \alpha + o(1)}.
\end{equation}
On the other hand, Lemma \ref{lemma:RK_low_degree} and a union bound over at most $n$ low-degree vertices implies that with probability $1 - o(1)$, 
\begin{equation}
\label{eq:RK_low_degree}
\min_{v \in V_K(T) : \deg(v) \le d} R_K(v) > n^{- \alpha /2},
\end{equation}
Together, \eqref{eq:RK_high_degree} and \eqref{eq:RK_low_degree} imply that if the threshold is set to be $h = n^{- \alpha/ 2}$, Algorithm \ref{alg:reinfections} indeed returns $\widehat{\mathrm{HD}} = \{ v \in V_K(T) : \deg(v) \ge n^\alpha \}$. \qed
\end{proof}

\subsection{Inference on short-term time horizons: Proof of Theorem \ref{thm:time_horizon}}

Suppose that $v$ is high degree and $I_1(v) \in [0,1]$. By Theorem \ref{thm:reinfections}, it suffices to show that $v \in V_K(T)$, or equivalently, $I_{K + 1}(v) \le T$, with probability at least $1 - \delta$. As a shorthand, let us define $X : = \sum_{k = 1}^K (S_k(v) - I_k(v))$ as well as the event $\cE : = \{ R_K(v) \le n^{- \alpha / 2 } \}$.
We remark that $X \sim \mathrm{Gamma}(K, \gamma)$ since each $S_k(v) - I_k(v)$ is an independent $\Exp(\gamma)$ random variable. Additionally, $\p ( \cE) = 1 - o(1)$ in light of Lemma \ref{lemma:RK_high_degree}.

Next, note that we have the decomposition
$I_{K + 1}(v) = I_1(v) + \sum_{k = 1}^K (I_{k + 1}(v)- S_k(v)) + X.$
On the event $\cE$,
\begin{equation}
\label{eq:IK_inequality}
| I_{K + 1}(v) - X | \le I_1(v) + K n^{- \alpha / 2} \le 2,
\end{equation}
where the first inequality holds on the event $\cE$, and the second inequality uses $I_1(v) \le 1$ and $K n^{- \alpha / 2} \le 1$ for $n$ sufficiently large.

Putting everything together, we have
\begin{multline*}
\p ( \{ I_{K + 1}(v) \ge T \} \cap \cE ) \stackrel{(a)}{\le} \p ( X \ge T -2 ) \\ \stackrel{(b)}{\le} \exp \left \{ K \log 2 - \frac{\gamma}{2} ( T - 2) \right \}  \stackrel{(c)}{\le} \delta,
\end{multline*}
where $(a)$ follows from \eqref{eq:IK_inequality}, $(b)$ is due to a Chernoff bound, and $(c)$ holds provided $T \ge c K \gamma^{-1} \log (1 / \delta)$, where $c > 0$ is a universal constant. Finally, since $\p ( \cE ) = 1 - o(1)$, it holds that $\p ( I_{K + 1}(v) \le T) \le \delta + o(1) \le 2 \delta$. 

We conclude by choosing a value of $K$. Theorem \ref{thm:reinfections} requires us to take $K > 2 / \alpha$, so we choose $K = \lceil 3 / \alpha \rceil$, which also leads to $T = \Omega (\alpha^{-1} \log (1 / \delta))$. \qed

\section{Simulations}
\label{sec:simulations}

\subsection{Network model and algorithm}

We simulate a stochastic SIS process on a graph $G$ with 1000 low-degree vertices and 10 additional high-degree vertices. Specifically, we first generate a random 4-regular graph $G'$ on 1000 vertices, and then add 10 vertices of degree $D$ whose neighbors are chosen uniformly at random from $G'$. The resulting graph is $G$. 
We choose this graph as a convenient case study for baseline analysis of our algorithm; we defer a study of more realistic networks to future work.
Infection and recovery rates are $\beta = 1.0$ and $\gamma = 0.5$; notably, the epidemic still persists in this case even when high-degree vertices are removed. The initial conditions are chosen so that 50\% of the network is infected at random. Additionally, all high-degree vertices are initially infected so that their impact on estimation and control can be clearly measured.
To estimate high-degree vertices, we use a modified version of Algorithm \ref{alg:reinfections}: for a given $K$, we compute $R_K(v)$ (see \eqref{eq:reinfection_time}) for all vertices $v$ and return the 10 vertices with the smallest values. This ranking-based approach is perhaps more natural when the number of high-degree vertices is known or there is a strict intervention budget.

\subsection{Estimation accuracy}

We first evaluate the accuracy of our algorithm as a function of $D$. For $D \in \{25, 50, 75, 100 \}$, we simulate the SIS process on $G$ for time horizons up to $T = 20$. We then set $K = 1$ for our algorithm, and computed the accuracy of its output, where accuracy is defined to be the fraction of true high-degree nodes identified. Our results are averaged over 50 trials, and displayed in Figure \ref{fig:accuracy}.
For all values of $D$, the accuracy-time curves exhibit similar characteristics. There is a brief initial increase (likely due to early rapid re-infections), followed by a short dip (likely from false positives as more re-infections occur throughout the network), and then a sharp exponential-type increase as the set of true high-degree vertices becomes clearly identifiable. As the observation time increases further, the accuracy approaches the maximum value of 1, though at different rates depending on $D$.

\begin{figure}[t]
  \centering
  \includegraphics[width=0.4\textwidth]{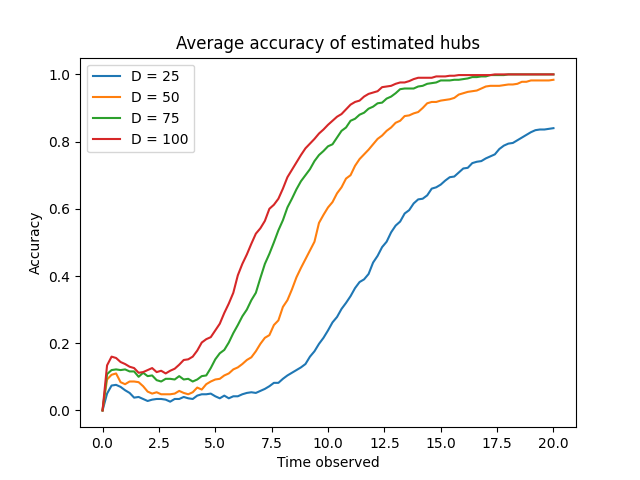}
  \caption{Average accuracy of estimating high-degree vertices for $K = 1$ and $D \in \{25, 50, 75, 100 \}$.}
  \label{fig:accuracy}
\end{figure}

We next examine the impact of parameter $K$ for $D = 25$ (Figure \ref{fig:DK_25}), averaging the accuracy over 50 trials for $K \in \{1,\dots,5\}$ and $T \le 20$. 
Interestingly, Figure \ref{fig:DK_25} reveals that the optimal choice of $K$ depends on the time horizon $T$: small $K$ yields higher accuracy for short horizons where re-infections are rare, while larger $K$ performs better in intermediate regimes. For long time horizons, the accuracy is less sensitive to $K$. We also compare our approach to a baseline that selects the ten vertices with the longest cumulative infection time; this strategy is known to identify high-degree vertices in the steady state regime $T \to \infty$ \cite{pastor2001epidemic}. While the baseline matches our algorithm's performance for large $T$, our methods obtain higher accuracy for small $T$. Notably, for $K \ge 3$, our algorithms meet or exceed the baseline accuracy across all time regimes.

\begin{figure}[t]
  \centering
  \begin{minipage}{0.5\textwidth}
    \centering
    \includegraphics[width=0.8\textwidth]{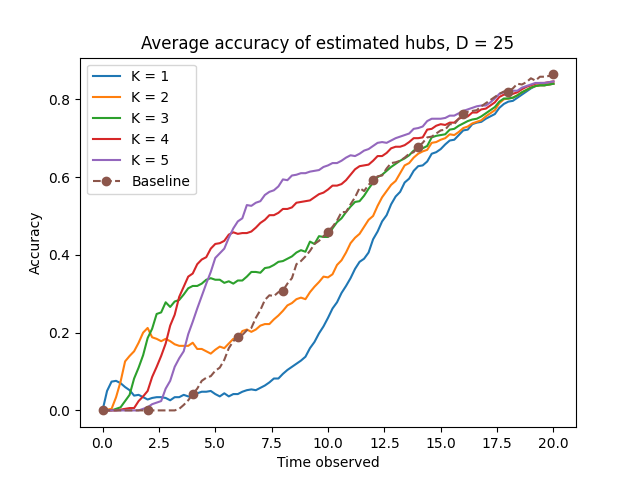}
    \caption{Average accuracy of estimating high-degree vertices for $K \in \{1,2,3,4,5 \}$ and $D = 25$}
    \label{fig:DK_25}
  \end{minipage}
  \label{fig:main}
\end{figure}

\subsection{Reduction in the number of infections}

\begin{figure}[t]
  \centering
  \includegraphics[width=0.4\textwidth]{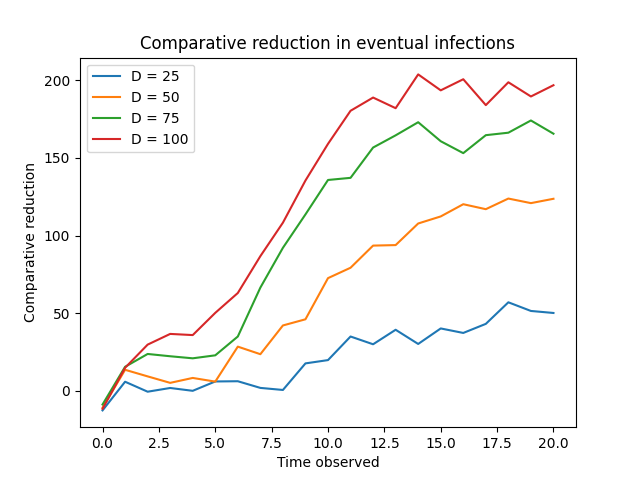}
  \caption{Comparative reduction in number of infections after interventions for various $D$, $K = 1$.}
  \label{fig:reduction}
\end{figure}

Finally, we illustrate how our algorithm for estimating high-degree vertices can inform interventions.
For observation windows $T$ between 0 and 20, we simulate the SIS process until time $T$. We then remove the estimated high-degree set (computed with $K = 1$) and simulate the SIS process for an additional 5.0 time units. Figure \ref{fig:reduction} shows the \emph{comparative reduction} in infections relative to removing a random set of 10 nodes (i.e., the number of infections after random removal minus the number of infections after removing nodes according to our algorithm), averaged over 40 trials. For short time horizons, little can be learned about high-degree vertices, and there is correspondingly little to no advantage of our methods over random selection. After $T \approx 2.5$, the benefits of targeted removals become more evident (at least for $D \ge 50$), with reductions growing as accuracy improves. Once the true high-degree set is reliably identifiable, the reductions plateau for long time horizons.

\section{Conclusion}
\label{sec:conclusion}

In this paper, we developed an algorithm for estimating high-degree vertices from the evolution of a SIS process on an unknown graph. Notably, our algorithm only require a constant-sized observation window to accurately estimate high-degree vertices, whereas the current state-of-the-art requires observing the process for time $T = \Omega(n)$ to learn the entire network \cite{elahi2025learn}. 
Major avenues of future work include identifying other topological features relevant for control such as community structure and degree distributions, as well as analyzing more realistic models for epidemics with re-infections.
\vspace{1cm} 
\bibliographystyle{IEEEtran}
\bibliography{IEEEabrv, references}

@InProceedings{mossel2024finding,
  title = 	 {Finding Super-spreaders in Network Cascades},
  author =       {Mossel, Elchanan and Sridhar, Anirudh},
  booktitle = 	 {Proceedings of Thirty Seventh Conference on Learning Theory},
  pages = 	 {3874--3914},
  year = 	 {2024},
  volume = 	 {247},
  series = 	 {Proceedings of Machine Learning Research},
  month = 	 {30 Jun--03 Jul},
  publisher =    {PMLR}
}

@article{WT04_epidemic,
    author = {Wallinga, Jacco and Teunis, Peter},
    title = {Different epidemic curves for severe acute respiratory syndrome reveal similar impacts of control measures},
    journal = {American Journal of Epidemiology},
    year = {2004},
    number={160},
    volume={6},
    pages ={509-516}
}

@inproceedings{AA05_blog,
author = {Adar, Eytan and Adamic, Lada A.},
title = {Tracking Information Epidemics in Blogspace},
year = {2005},
publisher = {IEEE Computer Society},
address = {USA},
booktitle={the Proceedings of the IEEE/WIC/ACM Conference on Web Intelligence},
pages = {207–214},
numpages = {8}
}

@article{gomez2012inferring,
author = {Gomez-Rodriguez, Manuel and Leskovec, Jure and Krause, Andreas},
title = {Inferring Networks of Diffusion and Influence},
year = {2012},
issue_date = {February 2012},
publisher = {Association for Computing Machinery},
address = {New York, NY, USA},
volume = {5},
number = {4},
issn = {1556-4681},
journal = {ACM Trans. Knowl. Discov. Data},
month = {feb},
articleno = {21},
numpages = {37}
}

@inproceedings{ACFKP13_trace_complexity,
author = {Abrahao, Bruno and Chierichetti, Flavio and Kleinberg, Robert and Panconesi, Alessandro},
title = {Trace complexity of network inference},
year = {2013},
publisher = {Association for Computing Machinery},
address = {New York, NY, USA},
booktitle = {the Proceedings of the 19th ACM SIGKDD International Conference on Knowledge Discovery and Data Mining},
pages = {491–499},
numpages = {9},
location = {Chicago, Illinois, USA}
}

@article{shen2014reconstructing,
    author = {Shen, Z. and Wang, W.X. and Fan, Y. and Di, Z. and Ying-Cheng, L.},
    title = {Reconstructing propagation networks with natural diversity and identifying hidden sources},
    journal = {Nature communications},
    year = {2014},
    volume = {5},
    number = {14}
}

@article{prasse2019infeasible,
    author = {Prasse, B. and Van Mieghem, P.},
    title = {Exact Network Reconstruction from Complete SIS Nodal State Infection Information Seems Infeasible},
    journal = {{IEEE} Transactions on Network Science and Engineering},
    year = {2019},
    volume = {6}, 
    number = {4},
    pages = {748 -- 759}
}

@article{brandenberger2025detecting,
    author = {Brandenberger, A. and Mossel, E. and Sridhar, A.},
    title = {Detecting Abrupt Changes in Point Processes: Fundamental Limits and Applications},
    year = {2025},
    note = {Preprint available at \url{https://arxiv.org/abs/2501.08392}}
}

@article{borgs2010antidote,
author = {Borgs, Christian and Chayes, Jennifer and Ganesh, Ayalvadi and Saberi, Amin},
title = {How to distribute antidote to control epidemics},
journal = {Random Structures \& Algorithms},
volume = {37},
number = {2},
pages = {204-222},
doi = {https://doi.org/10.1002/rsa.20315},
year = {2010}
}

@ARTICLE{nowzari2016analysis,
  author={Nowzari, Cameron and Preciado, Victor M. and Pappas, George J.},
  journal={IEEE Control Systems Magazine}, 
  title={Analysis and Control of Epidemics: A Survey of Spreading Processes on Complex Networks}, 
  year={2016},
  volume={36},
  number={1},
  pages={26-46},
  doi={10.1109/MCS.2015.2495000}}

@inproceedings{
elahi2025learn,
title={Learn to Vaccinate: Combining Structure Learning and Effective Vaccination for Epidemic and Outbreak Control},
author={Sepehr Elahi and Paula M{\"u}rmann and Patrick Thiran},
booktitle={Forty-second International Conference on Machine Learning},
year={2025}
}

@ARTICLE{zhang2024estimation,
  author={Zhang, Ciyuan and Leung, Humphrey and Butler, Brooks A. and Paré, Philip E.},
  journal={IEEE Transactions on Control of Network Systems}, 
  title={Estimation and Distributed Eradication of SIR Epidemics on Networks}, 
  year={2024},
  volume={11},
  number={2},
  pages={756-768},
  doi={10.1109/TCNS.2023.3306491}}

@ARTICLE{preciado2014optimal,
  author={Preciado, Victor M. and Zargham, Michael and Enyioha, Chinwendu and Jadbabaie, Ali and Pappas, George J.},
  journal={IEEE Transactions on Control of Network Systems}, 
  title={Optimal Resource Allocation for Network Protection Against Spreading Processes}, 
  year={2014},
  volume={1},
  number={1},
  pages={99-108},
  doi={10.1109/TCNS.2014.2310911}}

@article{dezso2002halting,
  title = {Halting viruses in scale-free networks},
  author = {Dezs\ifmmode \mbox{\H{o}}\else \H{o}\fi{}, Zolt\'an and Barab\'asi, Albert-L\'aszl\'o},
  journal = {Physical Review E},
  volume = {65},
  issue = {5},
  pages = {055103},
  numpages = {4},
  year = {2002},
  month = {May},
  publisher = {American Physical Society}
}

@article{pastor2002immunization,
  title = {Immunization of complex networks},
  author = {Pastor-Satorras, Romualdo and Vespignani, Alessandro},
  journal = {Phys. Rev. E},
  volume = {65},
  issue = {3},
  pages = {036104},
  numpages = {8},
  year = {2002},
  month = {Feb},
  publisher = {American Physical Society}
}

@book{contact_process,
title={Stochastic {I}nteracting {S}ystems: {C}ontact, {V}oter and {E}xclusion {P}rocesses},
author={Liggett, T.M.},
year={1999},
volume={324},
publisher={Springer, Berlin, Heidelberg}
}

@Book{brauer2012mathematical,
  title		= {{M}athematical {M}odels in {P}opulation {B}iology and
		  {E}pidemiology},
  author	= {Brauer, F. and Castillo-Chavez, C.},
  year		= {2012},
  publisher	= {Springer}
}

@article{newman2002spread,
  title = {Spread of epidemic disease on networks},
  author = {Newman, M. E. J.},
  journal = {Phys. Rev. E},
  volume = {66},
  issue = {1},
  pages = {016128},
  numpages = {11},
  year = {2002},
  month = {Jul},
  publisher = {American Physical Society},
  doi = {10.1103/PhysRevE.66.016128}
}

@article{salathe2010resolution,
author = {Marcel Salathé  and Maria Kazandjieva  and Jung Woo Lee  and Philip Levis  and Marcus W. Feldman  and James H. Jones },
title = {A high-resolution human contact network for infectious disease transmission},
journal = {Proceedings of the National Academy of Sciences},
volume = {107},
number = {51},
pages = {22020-22025},
year = {2010},
doi = {10.1073/pnas.1009094108},
eprint = {https://www.pnas.org/doi/pdf/10.1073/pnas.1009094108}}

@article{pastor2001epidemic,
  title = {Epidemic Spreading in Scale-Free Networks},
  author = {Pastor-Satorras, Romualdo and Vespignani, Alessandro},
  journal = {Phys. Rev. Lett.},
  volume = {86},
  issue = {14},
  pages = {3200--3203},
  numpages = {0},
  year = {2001},
  month = {Apr},
  publisher = {American Physical Society}
}

\end{document}